\documentclass[11pt,a4paper]{article}

\usepackage{inputenc}
\usepackage{amsmath}
\usepackage{bm}
\usepackage{bbold}
\usepackage{amsthm}
\usepackage{enumerate}

\usepackage{hyperref}

\title{Rating Alternatives from Pairwise Comparisons by Solving Tropical Optimization Problems}

\author{N. Krivulin\thanks{Faculty of Mathematics and Mechanics, Saint Petersburg State University, 28 Universitetsky Ave., St.~Petersburg, 198504, Russia, 
nkk@math.spbu.ru.}
}

\date{}

\newtheorem{theorem}{Theorem}
\newtheorem{lemma}[theorem]{Lemma}

\theoremstyle{definition}
\newtheorem{example}{Example}

\begin{document}

\maketitle

\begin{abstract}
We consider problems of rating alternatives based on their pairwise comparison under various assumptions, including constraints on the final scores of alternatives. The problems are formulated in the framework of tropical mathematics to approximate pairwise comparison matrices by reciprocal matrices of unit rank, and written in a common form for both multiplicative and additive comparison scales. To solve the unconstrained and constrained approximation problems, we apply recent results in tropical optimization, which provide new complete direct solutions given in a compact vector form. These solutions extend known results and involve less computational effort. As an illustration, numerical examples of rating alternatives are presented.
\\

\\

\end{abstract}

\section{Introduction}
Tropical (idempotent) mathematics, which studies idempotent semirings \cite{Kolokoltsov1997Idempotent,Golan2003Semirings,Heidergott2006Maxplus,Akian2007Maxplus,Litvinov2007Themaslov,Gondran2008Graphs,Speyer2009Tropical,Butkovic2010Maxlinear}, finds increasing applications in solving real-world problems in various fields, including decision making. To solve these problems in the framework of tropical mathematics, they are formulated as optimization problems to minimize or maximize functions defined on vectors over idempotent semifields (see, e.g., an overview in \cite{Krivulin2014Tropical}).

One of the applications of tropical optimization is concerned with the analysis of preferences by using pairwise comparison data in decision making. A problem of rating alternatives from their pairwise comparison matrix is examined in \cite{Elsner2004Maxalgebra,Elsner2010Maxalgebra,Tran2013Pairwise}. The problem is solved as an optimization problem in terms of tropical mathematics. The vector of final scores of alternatives is found as a tropical eigenvector of the matrix.

In this paper, we offer new solutions to the problems of rating alternatives under various assumptions, including constraints on the final scores of alternatives. The problems are formulated in the framework of tropical optimization to approximate pairwise comparison matrices by reciprocal matrices of unit rank. We consider unconstrained and constrained approximation of one matrix, and  simultaneous approximation of several matrices. The approximation problems are written in a common form for both multiplicative and additive comparison scales. To solve the problems, we apply recent results in \cite{Krivulin2014Aconstrained,Krivulin2015Amultidimensional,Krivulin2015Extremal}, which provide new complete, direct solutions given in a compact vector form. These new solutions extend known results and involve less computational effort. As an illustration, numerical examples of rating alternatives from pairwise comparison matrices are presented.

\section{Rating Alternatives from Pairwise Comparison}
\label{S-RAFPC}

Pairwise comparison techniques are widely used to obtain and arrange source data in the analysis of preferences in decision making (see, e.g., \cite{Thurstone1927Alaw,Saati1980Theanalytic,David1988Method}). Given results of pairwise comparison of alternatives on an appropriate scale, the analysis focuses on forming judgment on the overall preference of each alternative by evaluating its individual rating (score, priority).

\subsection{Pairwise Comparison Matrices}

The results of comparing alternatives in pairs with multiplicative or additive scales are described by pairwise comparison matrices that have a specific antisymmetric form. Let $\bm{A}=(a_{ij})$ be a pairwise comparison matrix. If the matrix is obtained on the basis of a multiplicative scale, then each entry $a_{ij}$ shows that alternative $i$ is preferred to $j$ by $a_{ij}$ times. The multiplicative comparison matrix is reciprocal, which means that its entries are positive and satisfy the condition
$$
a_{ij}
=
1/a_{ji}.
$$

In the case of additive scale, the entry $a_{ij}$ in the matrix $\bm{A}$ indicates by how many score units the preference of $i$ is greater than that of $j$. Then, the matrix $\bm{A}$ is skew-symmetric with the entries that answer the equality
$$
a_{ij}
=
-a_{ji}.
$$

In practice, the matrices composed of the results from pairwise comparisons on a multiplicative (additive) scale may be not reciprocal (skew-symmetric), and thus need corrections.  

To provide consistency and interpretability of the preference relation, the results of pairwise comparison have to be transitive, which implies that the entries of the multiplicative (additive) comparison matrix comply with the equality
$$
a_{ij}
=
a_{ik}a_{kj}
\qquad
(a_{ij}
=
a_{ik}+a_{kj}).
$$ 

A pairwise comparison matrix with transitive entries is called consistent. Every consistent matrix has a well-known form defined by a vector. Specifically, for any multiplicative (additive)  consistent matrix $\bm{A}=(a_{ij})$, there is a vector $\bm{x}=(x_{i})$ that completely specifies the entries of $\bm{A}$ as follows:
$$
a_{ij}
=
x_{i}/x_{j}
\qquad
(a_{ij}
=
x_{i}-x_{j}).
$$

At the same time, if a matrix $\bm{A}$ is consistent, then its corresponding vector $\bm{x}$ can be considered to specify directly the individual overall scores of alternatives, and thus provides the result, for which the analysis of preference is undertaken.

\subsection{Approximation by Consistent Matrices}

The matrices of pairwise comparison, which appear in real-world applications, are generally inconsistent, and may even have a non-antisymmetric form due to various reasons, from limitations in human judgment to data errors. This leads to a problem of approximating a matrix $\bm{A}$ obtained from pairwise comparisons by a consistent matrix $\bm{X}$, which is formulated to
\begin{equation}
\begin{aligned}
&
\text{minimize}
&&
\rho(\bm{A},\bm{X}),
\end{aligned}
\label{P-rhoAX}
\end{equation}
where the minimum is taken over all consistent matrices $\bm{X}$, and $\rho$ denotes a suitable measure of approximation error.

Since the entries of any consistent matrix $\bm{X}$ is uniquely determined by the elements of a vector $\bm{x}$, problem \eqref{P-rhoAX} is equivalent to finding this vector. Considering that the vector $\bm{x}$ shows the overall individual ratings of alternatives, the evaluation of preferences is reduced to the solution of \eqref{P-rhoAX}.

Several approaches exist to solve problem \eqref{P-rhoAX}, including approximation with the principal eigenvector of the matrix $\bm{A}$ \cite{Saati1980Theanalytic,Saaty1984Comparison}, least squares approximation \cite{Saaty1984Comparison,Chu1998Ontheoptimal} and other techniques \cite{Barzilai1997Deriving,Farkas2003Consistency,Gonzalezpachon2003Transitive}. As a rule, these approaches offer algorithmic solutions using iterative numerical procedures, such as power iterations in the principal eigenvector method and the Newton algorithm in the least squares approximation.

Another approach based on tropical mathematics is proposed and examined in \cite{Elsner2004Maxalgebra,Elsner2010Maxalgebra,Tran2013Pairwise}. This approach uses approximation by a consistent matrix formed by a tropical eigenvector, and hence can be considered a tropical counterpart of the conventional principal eigenvector method. Moreover, it is shown in \cite{Elsner2010Maxalgebra} that the matrix, which solves the approximation problem, can be defined not only by tropical eigenvectors, but also by some other vectors. A technique to find these vectors is proposed, which offers a computational algorithm, rather than provides a direct solution in an explicit form. 

Below, we formulate the problem of finding an approximate consistent matrix as a problem of approximation by reciprocal matrices of rank $1$ in the topical mathematics sense. We show how results in tropical optimization can be applied to provide a complete direct solution, and give numerical examples.

\section{Preliminary Definitions and Remarks}
\label{S-PDR}

In this section, we outline preliminary definitions and results of tropical mathematics from \cite{Krivulin2014Aconstrained,Krivulin2014Tropical,Krivulin2015Amultidimensional,Krivulin2015Extremal} to provide an appropriate analytical framework for the solutions in the subsequent sections. Further details at both introductory and advanced levels can be found, for instance, in \cite{Kolokoltsov1997Idempotent,Golan2003Semirings,Heidergott2006Maxplus,Akian2007Maxplus,Litvinov2007Themaslov,Gondran2008Graphs,Speyer2009Tropical,Butkovic2010Maxlinear}.

\subsection{Idempotent Semifield}

Let $\mathbb{X}$ be a set with two distinct elements $\bm{0}$ and $\bm{1}$, called the zero and the unit, and two binary operations $\oplus$ and $\otimes$, called addition and multiplication, such that $(\mathbb{X},\oplus,\bm{0})$ is an idempotent commutative monoid, $(\mathbb{X},\otimes,\bm{1})$ is an Abelian group, multiplication distributes over addition, and the zero is absorbing for multiplication. Under these conditions, the system $(\mathbb{X},\oplus,\otimes,\bm{0},\bm{1})$ is referred to as the idempotent semifield.

In the semifield, addition is idempotent to have $x\oplus x=x$ for all $x\in\mathbb{X}$. Multiplication is invertible, which means that each nonzero $x\in\mathbb{X}$ has its inverse $x^{-1}$ such that $x\otimes x^{-1}=\bm{1}$. The integer powers represent iterated products as $x^{0}=\bm{1}$, $x^{p}=x^{p-1}x$, $x^{-p}=(x^{-1})^{p}$ for any $x\ne\bm{0}$ and integer $p>0$.

The semifield is assumed to have a linear order that is consistent with the partial order induced by idempotent addition to define $x\leq y$ if and only if $x\oplus y=y$. Moreover, the semifield is considered algebraically closed (radicable), which means that the equation $x^{p}=a$ has solutions for any $a\ne\bm{0}$ and integer $p$ to provide the powers with rational exponents.

In the expressions that follow, the multiplication sign is usually omitted for the sake of brevity. 

Examples of the idempotent semifield under study include
\begin{align*}
\mathbb{R}_{\max,\times}
&=
(\mathbb{R}_{+}\cup\{0\},\max,\times,0,1),
\\
\mathbb{R}_{\max,+}
&=
(\mathbb{R}\cup\{-\infty\},\max,+,-\infty,0),
\end{align*}
where $\mathbb{R}$ is the set of real numbers, and $\mathbb{R}_{+}=\{x\in\mathbb{R}|x>0\}$.

The semifield $\mathbb{R}_{\max,\times}$ has the addition $\oplus$ defined as maximum, and the multiplication $\otimes$ defined as usual. The neutral elements $\bm{0}$ and $\bm{1}$ coincide with the arithmetic zero and one. The power and inversion notation has the standard meaning. 

The semifield $\mathbb{R}_{\max,+}$ is equipped with $\oplus=\max$, $\otimes=+$, $\bm{0}=-\infty$ and $\bm{1}=0$. For each $x\in\mathbb{R}$, the inverse $x^{-1}$ coincides with the usual opposite number $-x$. For all $x,y\in\mathbb{R}$, the power $x^{y}$ corresponds to the regular arithmetic product $xy$.

In both semifields, the idempotent addition induces the order, which is consistent with the natural linear order on $\mathbb{R}$.

\subsection{Vector and Matrix Algebra}

The set of column vectors with $n$ elements over $\mathbb{X}$ is denoted $\mathbb{X}^{n}$. A vector with all elements equal to $\bm{0}$ is the zero vector. A vector is called regular if it has no zero elements. Vector addition and scalar multiplication follow the conventional element-wise rules, where the scalar operations $\oplus$ and $\otimes$ play the roles of the standard addition and multiplication.

A vector $\bm{b}$ is linearly dependent on vectors $\bm{a}_{1},\ldots,\bm{a}_{m}$ if $\bm{b}=x_{1}\bm{a}_{1}\oplus\cdots\oplus x_{m}\bm{a}_{m}$ for some scalars $x_{1},\ldots,x_{m}$.
Vectors $\bm{a}$ and $\bm{b}$ are collinear if $\bm{b}=x\bm{a}$ for some scalar $x$. The set of linear combinations $x_{1}\bm{a}_{1}\oplus\cdots\oplus x_{m}\bm{a}_{m}$ for all possible coefficients $x_{1},\ldots,x_{m}$ is closed under vector addition and scalar multiplication, and is referred to as the idempotent vector space generated by the vectors $\bm{a}_{1},\ldots,\bm{a}_{m}$.

For each nonzero column vector $\bm{x}=(x_{i})$, the multiplicative conjugate transpose is the row vector $\bm{x}^{-}=(x_{i}^{-})$ with the elements $x_{i}^{-}=x_{i}^{-1}$ if $x_{i}\ne\bm{0}$, and $x_{i}^{-}=\bm{0}$ otherwise.

The matrices with $m$ rows and $n$ columns form the set $\mathbb{X}^{m\times n}$. Matrix addition, matrix multiplication and scalar multiplication are routinely defined entry-wise, where the operations $\oplus$ and $\otimes$ are used instead of the usual addition and multiplication.

For any nonzero matrix $\bm{A}=(a_{ij})\in\mathbb{X}^{m\times n}$, the multiplicative conjugate transpose is the matrix $\bm{A}^{-}=(a_{ij}^{-})\in\mathbb{X}^{n\times m}$, where $a_{ij}^{-}=a_{ji}^{-1}$ if $a_{ji}\ne\bm{0}$, and $a_{ij}^{-}=\bm{0}$ otherwise.

The rank of a matrix is defined as the maximum number of linearly independent columns (rows) in the matrix. A matrix $\bm{A}$ has rank $1$ if and only if there exist nonzero column vectors $\bm{x}$ and $\bm{y}$ such that $\bm{A}=\bm{x}\bm{y}^{T}$. 

Consider square matrices of order $n$ in the set $\mathbb{X}^{n\times n}$. A matrix that has $\bm{1}$ along the diagonal, and $\bm{0}$ elsewhere, is the identity matrix denoted $\bm{I}$. The power notation is defined to indicate repeated multiplication as $\bm{A}^{0}=\bm{I}$ and $\bm{A}^{p}=\bm{A}^{p-1}\bm{A}$ for any square matrix $\bm{A}$ and integer $p>0$.

A matrix $\bm{A}$ without zero entries is called symmetrically reciprocal (or, simply, reciprocal) if the condition $\bm{A}^{-}=\bm{A}$ holds. A reciprocal matrix $\bm{A}$ is of unit rank if and only if $\bm{A}=\bm{x}\bm{x}^{-}$, where $\bm{x}$ is a regular column vector.

The trace of a matrix $\bm{A}=(a_{ij})$ is given by
$$
\mathop\mathrm{tr}\bm{A}
=
a_{11}\oplus\cdots\oplus a_{nn}.
$$
For any matrices $\bm{A}$ and $\bm{B}$, the following equalities hold:
$$
\mathop\mathrm{tr}(\bm{A}\oplus\bm{B})
=
\mathop\mathrm{tr}\bm{A}
\oplus
\mathop\mathrm{tr}\bm{B},
\qquad
\mathop\mathrm{tr}(\bm{A}\bm{B})
=
\mathop\mathrm{tr}(\bm{B}\bm{A}).
$$

To describe solutions of optimization problems below, a function is used that takes any matrix $\bm{A}$ to produce the scalar
$$
\mathop\mathrm{Tr}(\bm{A})
=
\mathop\mathrm{tr}\bm{A}
\oplus\cdots\oplus
\mathop\mathrm{tr}\bm{A}^{n}.
$$

Provided that $\mathop\mathrm{Tr}(\bm{A})\leq\bm{1}$, the star operator (also known as the Kleene star) maps $\bm{A}$ into the matrix
$$
\bm{A}^{\ast}
=
\bm{I}\oplus\bm{A}\oplus\cdots\oplus\bm{A}^{n-1}.
$$

\subsection{Distance Functions}

The distance between two regular vectors $\bm{x},\bm{y}\in\mathbb{X}^{n}$ is given by the function
$$
\rho(\bm{x},\bm{y})
=
\bm{y}^{-}\bm{x}
\oplus
\bm{x}^{-}\bm{y},
$$
which takes the minimum value $\bm{1}$ only when $\bm{y}=\bm{x}$. 

For the real semifield $\mathbb{R}_{\max,+}$, where $\bm{1}=0$, the function $\rho$ coincides with the Chebyshev metric
$$
\rho_{\infty}(\bm{x},\bm{y})
=
\max_{1\leq i\leq n}|x_{i}-y_{i}|.
$$

In the case of $\mathbb{R}_{\max,\times}$, the function $\rho$ differs from the usual metrics in the range of values, and becomes a log-Chebyshev metric after taking the logarithm. In the general case, the function $\rho$ is referred to as the Chebyshev-like distance.

The distance between two matrices without zero entries is defined by the Chebyshev-like distance function
\begin{equation}
\rho(\bm{A},\bm{B})
=
\mathop\mathrm{tr}(\bm{B}^{-}\bm{A})
\oplus
\mathop\mathrm{tr}(\bm{A}^{-}\bm{B}).
\label{E-rhoAB}
\end{equation}

This function has the form of the Chebyshev metric for the semifield $\mathbb{R}_{\max,+}$, and takes the form of a log-Chebyshev metric after logarithmic transformation for $\mathbb{R}_{\max,\times}$.

\subsection{Eigenvalues and Eigenvectors of Matrices}

A scalar $\lambda\in\mathbb{X}$ is an eigenvalue of a matrix $\bm{A}\in\mathbb{X}^{n\times n}$ if there exists a nonzero vector $\bm{x}\in\mathbb{X}^{n}$ such that $\bm{A}\bm{x}=\lambda\bm{x}$. This vector $\bm{x}$ is an eigenvector of $\bm{A}$, corresponding to $\lambda$.

The maximum eigenvalue of a matrix $\bm{A}=(a_{ij})$ is referred to as the spectral radius of the matrix and calculated as
$$
\lambda
=
\mathop\mathrm{tr}\nolimits\bm{A}\oplus\cdots\oplus\mathop\mathrm{tr}\nolimits^{1/n}(\bm{A}^{n}),
$$
or, in terms of matrix entries, as
\begin{equation}
\lambda
=
\bigoplus_{k=1}^{n}\bigoplus_{1\leq i_{1},\ldots,i_{k}\leq n}(a_{i_{1}i_{2}}a_{i_{2}i_{3}}\cdots a_{i_{k}i_{1}})^{1/k}.
\label{E-lambda-ai1i2ai2i3aiki1}
\end{equation}

Any matrix $\bm{A}$ with nonzero entries has only one eigenvalue $\lambda$ given by the above expressions. The eigenvectors of $\bm{A}$, which correspond to $\lambda$, are derived as follows. For the matrix $\bm{A}_{\lambda}=\lambda^{-1}\bm{A}$, calculate the matrix (the Kleene star)
$$
\bm{A}_{\lambda}^{\ast}
=
\bm{I}\oplus\bm{A}_{\lambda}\oplus\cdots\oplus\bm{A}_{\lambda}^{n-1},
$$
and then the matrix $\bm{A}_{\lambda}^{\times}=\bm{A}_{\lambda}\bm{A}_{\lambda}^{\ast}$. It remains to form the matrix $\bm{A}_{\lambda}^{+}$ by taking those columns that coincide in both matrices $\bm{A}_{\lambda}^{\ast}$ and $\bm{A}_{\lambda}^{\times}$. All eigenvectors of $\bm{A}$ are given by
$$
\bm{x}
=
\bm{A}_{\lambda}^{+}\bm{u},
$$
where $\bm{u}$ is any vector of appropriate size, and hence constitute an idempotent vector space generated by the columns in $\bm{A}_{\lambda}^{+}$.

\section{Tropical Optimization Problems}
\label{S-TOP}

We now consider optimization problems that are formulated and solved in the tropical mathematics setting. We start with the unconstrained problem: given a matrix $\bm{A}\in\mathbb{X}^{n\times n}$, find regular vectors $\bm{x}\in\mathbb{X}^{n}$ that
\begin{equation}
\begin{aligned}
&
\text{minimize}
&&
\bm{x}^{-}\bm{A}\bm{x},
\end{aligned}
\label{P-minxAx}
\end{equation}

A complete, direct solution to this problem is obtained in \cite{Krivulin2014Aconstrained,Krivulin2015Amultidimensional,Krivulin2015Extremal} in the following form.
\begin{lemma}
\label{L-minxAx}
Let $\bm{A}$ be a matrix with spectral radius $\lambda>\bm{0}$, and $\bm{A}_{\lambda}=\lambda^{-1}\bm{A}$. Then, the minimum value in problem \eqref{P-minxAx} is equal to $\lambda$, and all regular solutions are given by
$$
\bm{x}
=
\bm{A}_{\lambda}^{\ast}\bm{u},
\qquad
\bm{u}\in\mathbb{X}^{n}.
$$
\end{lemma}

It follows from Lemma~\ref{L-minxAx} that the solutions form an idempotent vector space spanned by the columns of $\bm{A}_{\lambda}^{\ast}$.

Furthermore, suppose that, given matrices $\bm{A},\bm{C}\in\mathbb{X}^{n\times n}$, we need to find regular solutions $\bm{x}\in\mathbb{X}^{n}$ to the problem
\begin{equation}
\begin{aligned}
&
\text{minimize}
&&
\bm{x}^{-}\bm{A}\bm{x},
\\
&
\text{subject to}
&&
\bm{C}\bm{x}
\leq
\bm{x}.
\end{aligned}
\label{P-xAxCxx}
\end{equation}

The next complete solution to the problem is given in \cite{Krivulin2015Amultidimensional}.

\begin{theorem}
\label{T-xAxCxx}
Let $\bm{A}$ be a matrix with spectral radius $\lambda>\bm{0}$, and $\bm{C}$ a matrix such that $\mathop\mathrm{Tr}(\bm{C})\leq\bm{1}$. Then, the minimum value in problem \eqref{P-xAxCxx} is equal to
\begin{equation*}
\theta
=
\lambda
\oplus
\bigoplus_{k=1}^{n-1}\mathop{\bigoplus\hspace{1.1em}}_{1\leq i_{1}+\cdots+i_{k}\leq n-k}\mathop\mathrm{tr}\nolimits^{1/k}(\bm{A}\bm{C}^{i_{1}}\cdots\bm{A}\bm{C}^{i_{k}}),
\end{equation*}
and all regular solutions are given by
\begin{equation*}
\bm{x}
=
(\theta^{-1}\bm{A}\oplus\bm{C})^{\ast}\bm{u},
\qquad
\bm{u}\in\mathbb{X}^{n}.
\end{equation*}
\end{theorem}

In the ensuing section, the above solutions are applied to solve matrix approximation problems, which appear in rating alternatives on the basis of their pairwise comparisons.

\section{Evaluation of Scores by Pairwise Comparisons}
\label{S-ESPC}

We are now in a position to put the problem of rating alternatives via approximation by consistent matrices in the context of tropical optimization. First, we note that, in the framework of tropical mathematics, both multiplicative and additive consistent matrices can be represented in a common form of the reciprocal matrix of rank $1$, which are given by
$$
\bm{X}
=
\bm{x}\bm{x}^{-},
$$
to be interpreted in terms of either the semifield $\mathbb{R}_{\max,\times}$ for the multiplicative case and the semifield $\mathbb{R}_{\max,+}$ for the additive.

The problem of finding an approximate consistent matrix $\bm{X}$, or, equivalently, a vector of scores $\bm{x}$, takes the form: given a matrix $\bm{A}\in\mathbb{X}^{n\times n}$, find regular vectors $\bm{x}\in\mathbb{X}^{n}$ that
\begin{equation}
\begin{aligned}
&
\text{minimize}
&&
\rho(\bm{A},\bm{x}\bm{x}^{-}),
\end{aligned}
\label{P-rhoAxx}
\end{equation}
where $\rho$ is a measure of approximation error, which is given by the Chebyshev-like distance function defined as \eqref{E-rhoAB}. The function $\rho$ becomes a log-Chebyshev metric for the  multiplicative scale, and the Chebyshev metric for the additive.

In this section, we apply the solutions of tropical optimization problems given by Lemma~\ref{L-minxAx} and Theorem~\ref{T-xAxCxx} to approximation problem \eqref{P-rhoAxx} to evaluate, under various assumptions, the scores of alternatives, based on pairwise comparison matrices. The results obtained offer complete, direct solutions given in compact vector form, which extend the known solutions to the problems in \cite{Elsner2004Maxalgebra,Elsner2010Maxalgebra}, and are easier to calculate.  

Below, we examine problems of rating alternatives on a multiplicative scale. Considering that, in terms of tropical mathematics, the solution to approximation problems in both multiplicative and additive settings has a common general form, the case of additive scale is not covered here for brevity.

\subsection{Evaluation of Scores Given by One Matrix}

We first provide a solution for evaluating the vector of scores $\bm{x}\in\mathbb{R}_{+}^{n}$ on the basis of one pairwise comparison matrix $\bm{A}\in\mathbb{R}_{+}^{n\times n}$. The problem is described in the setting of the semifield $\mathbb{R}_{\max,\times}$ in the form of \eqref{P-rhoAxx} to approximate the matrix $\bm{A}$ by a reciprocal matrix of unit rank.
\begin{theorem}
\label{T-minrhoAxx}
Let $\bm{A}$ be a matrix such that the matrix $\bm{B}=\bm{A}\oplus\bm{A}^{-}$ has no zero entries, $\mu$ be the spectral radius of $\bm{B}$, and $\bm{B}_{\mu}=\mu^{-1}\bm{B}$. Then the minimum value in problem \eqref{P-rhoAxx} is equal to $\mu$, and all solutions are given by
$$
\bm{x}
=
\bm{B}_{\mu}^{\ast}\bm{u},
\qquad
\bm{u}
\in
\mathbb{R}_{+}^{n}.
$$
\end{theorem}
\begin{proof}
It is easy to see that, since all entries in the matrix $\bm{B}$ are nonzero, this matrix has the spectral radius $\mu>\bm{0}$. 

We use formula \eqref{E-rhoAB}, the equality $(\bm{x}\bm{x}^{-})^{-}=\bm{x}\bm{x}^{-}$, and properties of the trace to write the objective function as
\begin{multline*}
\rho(\bm{A},\bm{x}\bm{x}^{-})
=
\mathop\mathrm{tr}((\bm{x}\bm{x}^{-})^{-}\bm{A})
\oplus
\mathop\mathrm{tr}(\bm{A}^{-}\bm{x}\bm{x}^{-})
\\
=
\bm{x}^{-}\bm{A}\bm{x}
\oplus
\bm{x}^{-}\bm{A}^{-}\bm{x}
=
\bm{x}^{-}\bm{B}\bm{x}.
\end{multline*}

An application of Lemma~\ref{L-minxAx} completes the proof.
\end{proof}

We now give an example of evaluating the score vector from a reciprocal matrix of pairwise comparisons. For arbitrary positive matrices, evaluation of scores follows the same way.

\begin{example}
\label{E-Areciprocal}
Consider the reciprocal matrix defined as
$$
\bm{A}
=
\left(
\begin{array}{cccc}
1 & 3 & 2 & 4
\\
1/3 & 1 & 1/3 & 1/2
\\
1/2 & 3 & 1 & 1/4
\\
1/4 & 2 & 4 & 1
\end{array}
\right).
$$

To approximate the matrix by a reciprocal matrix of unit rank, and thus to find a score vector $\bm{x}$, we apply Theorem~\ref{T-minrhoAxx}. Since the matrix $\bm{A}$ is reciprocal, and hence $\bm{A}^{-}=\bm{A}$, we see that $\bm{B}=\bm{A}\oplus\bm{A}^{-}=\bm{A}$. We apply \eqref{E-lambda-ai1i2ai2i3aiki1} to find the spectral radius of the matrix $\bm{B}$ to be $\mu=2$.

Furthermore, we take the matrix
$$
\bm{B}_{\mu}
=
\mu^{-1}\bm{B}
=
\left(
\begin{array}{cccc}
1/2 & 3/2 & 1 & 2
\\
1/6 & 1/2 & 1/6 & 1/4
\\
1/4 & 3/2 & 1/2 & 1/8
\\
1/8 & 1 & 2 & 1/2
\end{array}
\right),
$$
and then calculate the powers
\begin{align*}
\bm{B}_{\mu}^{2}
&=
\left(
\begin{array}{cccc}
1/4 & 2 & 4 & 1
\\
1/12 & 1/4 & 1/2 & 1/3
\\
1/4 & 3/4 & 1/4 & 1/2
\\
1/2 & 3 & 1 & 1/4
\end{array}
\right),
\\
\bm{B}_{\mu}^{3}
&=
\left(
\begin{array}{cccc}
1 & 6 & 2 & 1/2
\\
1/8 & 3/4 & 2/3 & 1/6
\\
1/8 & 1/2 & 1 & 1/2
\\
1/2 & 3/2 & 1/2 & 1
\end{array}
\right).
\end{align*}

Finally, we compose the matrix
$$
\bm{B}_{\mu}^{\ast}
=
\bm{I}\oplus\bm{B}_{\mu}\oplus\bm{B}_{\mu}^{2}\oplus\bm{B}_{\mu}^{3}
=
\left(
\begin{array}{cccc}
1 & 6 & 4 & 2
\\
1/6 & 1 & 2/3 & 1/3
\\
1/4 & 3/2 & 1 & 1/2
\\
1/2 & 3 & 2 & 1
\end{array}
\right).
$$

Clearly, the columns in the matrix $\bm{B}_{\mu}^{\ast}$ are collinear to each other. Specifically, the last three columns can be obtained by multiplying the first one by $6$, $4$ and $2$, respectively. Since each column generates exactly the same vector space, it is sufficient to take only one of them to describe all solution vectors $\bm{x}$. We use the first column and write the score vector as
$$
\bm{x}
=
\left(
\begin{array}{c}
1
\\
1/6
\\
1/4
\\
1/2
\end{array}
\right)u,
$$
where $u$ is an arbitrary positive number to be fixed in accordance with the required form or interpretation of the result.

With $u=1$, the vector $\bm{x}=(1,1/6,1/4,1/2)^{T}$ shows that the first alternative is of the highest score $x_{1}=1$, followed by the fourth and third with scores $x_{4}=1/2$ and $x_{3}=1/4$. The second alternative has the lowest score $x_{2}=1/6$. If the scores are considered as weights, which must add up to one, we put $u=1/(1+1/6+1/4+1/2)=12/23$. The vector takes the form $\bm{x}=(12/23,2/23,3/23,6/23 )^{T}$.  
\end{example}

\subsection{Evaluation of Scores Given by Several Matrices}

Suppose that there are $m$ matrices $\bm{A}_{1},\ldots,\bm{A}_{m}\in\mathbb{X}^{n\times n}$, and we need to determine a reciprocal matrix of rank $1$ that approximates these matrices simultaneously. The approximation problem is defined in terms of the semifield $\mathbb{R}_{\max,\times}$ in a similar form as \eqref{P-rhoAxx} to find regular vectors $\bm{x}$ that
\begin{equation}
\begin{aligned}
&
\text{minimize}
&&
\max_{1\leq i\leq m}\rho(\bm{A}_{i},\bm{x}\bm{x}^{-}).
\end{aligned}
\label{P-rhomaxAixx}
\end{equation}

\begin{theorem}
\label{T-rhomaxAixx}
Let $\bm{A}_{i}$ be matrices for all $i=1,\ldots,m$ such that the matrix $\bm{B}=\bm{A}_{1}\oplus\bm{A}_{1}^{-}\oplus\cdots\oplus\bm{A}_{m}\oplus\bm{A}_{m}^{-}$ has no zero entries, $\mu$ be the spectral radius of $\bm{B}$, and $\bm{B}_{\mu}=\mu^{-1}\bm{B}$. Then, the minimum value in problem \eqref{P-rhomaxAixx} is equal to $\mu$, and all solutions are given by
$$
\bm{x}
=
\bm{B}_{\mu}^{\ast}\bm{u},
\qquad
\bm{u}
\in
\mathbb{R}_{+}^{n}.
$$
\end{theorem}
\begin{proof}
For each $i=1,\ldots,m$, we use the same argument as in Theorem~\ref{T-minrhoAxx} to write $\rho(\bm{A}_{i},\bm{x}\bm{x}^{-})=\bm{x}^{-}(\bm{A}_{i}\oplus\bm{A}_{i}^{-})\bm{x}$. Then, we represent the objective function as
$$
\max_{1\leq i\leq m}\rho(\bm{A}_{i},\bm{x}\bm{x}^{-})
=
\bigoplus_{i=1}^{m}\bm{x}^{-}(\bm{A}_{i}\oplus\bm{A}_{i}^{-})\bm{x}
=
\bm{x}^{-}\bm{B}\bm{x}.
$$

The desired result immediately follows from Lemma~\ref{L-minxAx}.
\end{proof}

\begin{example}
We now evaluate the score vector based on the simultaneous approximation of $m=2$ reciprocal matrices
\begin{align*}
\bm{A}_{1}
&=
\left(
\begin{array}{cccc}
1 & 3 & 2 & 4
\\
1/3 & 1 & 1/3 & 1/2
\\
1/2 & 3 & 1 & 1/3
\\
1/4 & 2 & 3 & 1
\end{array}
\right),
\\
\bm{A}_{2}
&=
\left(
\begin{array}{cccc}
1 & 4 & 2 & 3
\\
1/4 & 1 & 1/2 & 1/2
\\
1/2 & 2 & 1 & 1/4
\\
1/3 & 2 & 4 & 1
\end{array}
\right).
\end{align*}

To solve the problem by applying Theorem~\ref{T-rhomaxAixx}, we have to compose the matrix $\bm{B}=\bm{A}_{1}\oplus\bm{A}_{1}^{-}\oplus\bm{A}_{2}\oplus\bm{A}_{2}^{-}$. Considering that the matrices $\bm{A}_{1}$ and $\bm{A}_{2}$ are reciprocal, we have
$$
\bm{B}
=
\bm{A}_{1}\oplus\bm{A}_{2}
=
\left(
\begin{array}{cccc}
1 & 4 & 2 & 4
\\
1/3 & 1 & 1/2 & 1/2
\\
1/2 & 3 & 1 & 1/3
\\
1/3 & 2 & 4 & 1
\end{array}
\right).
$$

Note that the obtained matrix coincides with the matrix $\bm{B}$ in Example~\ref{E-Areciprocal}. Since the matrix $\bm{B}$ completely determines the set of solution vectors, we can use the result of this example, which offers the score vector $\bm{x}=(1,1/6,1/4,1/2)^{T}$. 
\end{example}

\subsection{Constrained Evaluation of Scores}

Consider the problem of evaluating the vector $\bm{x}=(x_{i})$, which represents the individual overall scores calculated from the results of pairwise comparison given by a matrix $\bm{A}$. Suppose that, for some reasons, additional constraints are imposed on the scores by inequalities in the form $x_{i}\geq c_{ij}x_{j}$, which requires that the overall score of alternative $i$ must be $c_{ij}$ times greater or more than the score of alternative $j$. 

To describe the problem in terms of tropical mathematics, we introduce a matrix $\bm{C}=(c_{ij})$, where we put $c_{ij}=0$ if no constraint is defined for alternatives $i$ and $j$. It is not difficult to see that the constraints can be represented as the vector inequality $\bm{C}\bm{x}\leq\bm{x}$ written in terms of the semifield $\mathbb{R}_{\max,\times}$.

By combining the constraint with the objective function, we arrive at the next constrained approximation problem in the framework of $\mathbb{R}_{\max,\times}$. Given matrices $\bm{A}$ and $\bm{C}$, the problem is to find regular vectors $\bm{x}$ that
\begin{equation}
\begin{aligned}
&
\text{minimize}
&&
\rho(\bm{A},\bm{x}\bm{x}^{-}),
\\
&
\text{subject to}
&&
\bm{C}\bm{x}
\leq
\bm{x}.
\end{aligned}
\label{P-rhoAxx-Cxx}
\end{equation}

\begin{theorem}
\label{T-rhoAxx-Cxx}
Let $\bm{A}$ be a matrix such that the matrix $\bm{B}=\bm{A}\oplus\bm{A}^{-}$ has no zero entries, $\mu$ be the spectral radius of $\bm{B}$, and $\bm{C}$ a matrix such that $\mathop\mathrm{Tr}(\bm{C})\leq\bm{1}$. Then, the minimum value in problem \eqref{P-rhoAxx-Cxx} is equal to
\begin{equation*}
\theta
=
\mu
\oplus
\bigoplus_{k=1}^{n-1}\mathop{\bigoplus\hspace{1.1em}}_{1\leq i_{1}+\cdots+i_{k}\leq n-k}\mathop\mathrm{tr}\nolimits^{1/k}(\bm{B}\bm{C}^{i_{1}}\cdots\bm{B}\bm{C}^{i_{k}}),
\end{equation*}
and all regular solutions are given by
\begin{equation*}
\bm{x}
=
(\theta^{-1}\bm{B}\oplus\bm{C})^{\ast}\bm{u},
\qquad
\bm{u}
\in
\mathbb{R}_{+}^{n}.
\end{equation*}
\end{theorem}
\begin{proof}
To verify the statement, we rewrite the objective function as in Theorem~\ref{T-minrhoAxx}, and then apply Theorem~\ref{T-xAxCxx}.
\end{proof}

\begin{example}
Let us evaluate scores in a constrained problem, where the results of pairwise comparison and the constraints are defined by the matrices
$$
\bm{A}
=
\left(
\begin{array}{cccc}
1 & 3 & 2 & 4
\\
1/3 & 1 & 1/3 & 1/2
\\
1/2 & 3 & 1 & 1/4
\\
1/4 & 2 & 4 & 1
\end{array}
\right),
\quad
\bm{C}
=
\left(
\begin{array}{cccc}
0 & 0 & 0 & 0
\\
0 & 0 & 0 & 1
\\
0 & 1 & 0 & 0
\\
0 & 0 & 1 & 0
\end{array}
\right).
$$

Note that the solution to the unconstrained problem with the matrix $\bm{B}=\bm{A}$ is provided by Example~\ref{E-Areciprocal}. Furthermore, the constraints $\bm{C}\bm{x}\leq\bm{x}$ given by the matrix $\bm{C}$ take the form
$$
x_{4}
\leq
x_{2},
\qquad
x_{2}
\leq
x_{3},
\qquad
x_{3}
\leq
x_{4},
$$
which is obviously equivalent to one condition $x_{2}=x_{3}=x_{4}$.

By Theorem~\ref{T-rhoAxx-Cxx}, we have to calculate the value of $\theta$. Using properties of the trace yields the expression
\begin{multline*}
\theta
=
\mu
\oplus
\mathop\mathrm{tr}\nolimits(\bm{B}\bm{C}(\bm{I}\oplus\bm{C}\oplus\bm{C}^{2}))
\\
\oplus
\mathop\mathrm{tr}\nolimits^{1/2}(\bm{B}\bm{C}\bm{B}(\bm{I}\oplus\bm{C}))
\oplus
\mathop\mathrm{tr}\nolimits^{1/3}(\bm{B}\bm{C}\bm{B}^{2}),
\end{multline*}
where $\mu=2$ is the spectral radius of the matrix $\bm{B}$.

We now calculate the matrices
\begin{gather*}
\bm{I}\oplus\bm{C}
=
\left(
\begin{array}{cccc}
1 & 0 & 0 & 0
\\
0 & 1 & 0 & 1
\\
0 & 1 & 1 & 0
\\
0 & 0 & 1 & 1
\end{array}
\right),
\quad
\bm{C}^{2}
=
\left(
\begin{array}{cccc}
0 & 0 & 0 & 0
\\
0 & 0 & 1 & 0
\\
0 & 0 & 0 & 1
\\
0 & 1 & 0 & 0
\end{array}
\right),
\\
\bm{I}\oplus\bm{C}\oplus\bm{C}^{2}
=
\left(
\begin{array}{cccc}
1 & 0 & 0 & 0
\\
0 & 1 & 1 & 1
\\
0 & 1 & 1 & 1
\\
0 & 1 & 1 & 1
\end{array}
\right).
\end{gather*}

Next, we obtain the matrices
\begin{gather*}
\bm{B}\bm{C}
=
\left(
\begin{array}{cccc}
0 & 2 & 4 & 3
\\
0 & 1/3 & 1/2 & 1 
\\
0 & 1 & 1/4 & 3
\\
0 & 4 & 1 & 2
\end{array}
\right),
\\
\bm{B}\bm{C}(\bm{I}\oplus\bm{C}\oplus\bm{C}^{2})
=
\left(
\begin{array}{cccc}
0 & 4 & 4 & 4
\\
0 & 1 & 1 & 1
\\
0 & 3 & 3 & 3
\\
0 & 4 & 4 & 4
\end{array}
\right),
\end{gather*}
and then find $\mathop\mathrm{tr}\nolimits(\bm{B}\bm{C}(\bm{I}\oplus\bm{C}\oplus\bm{C}^{2}))=4$.

Furthermore, we calculate
\begin{gather*}
\bm{B}\bm{C}\bm{B}
=
\left(
\begin{array}{cccc}
2 & 12 & 12 & 3
\\
1/4 & 2 & 4 & 1 
\\
3/4 & 6 & 12 & 3
\\
4/3 & 4 & 8 & 2
\end{array}
\right),
\\
\bm{B}\bm{C}\bm{B}(\bm{I}\oplus\bm{C})
=
\left(
\begin{array}{cccc}
2 & 12 & 12 & 12
\\
1/4 & 4 & 4 & 2 
\\
3/4 & 12 & 12 & 6
\\
4/3 & 8 & 8 & 4
\end{array}
\right),
\end{gather*}
from which it follows that $\mathop\mathrm{tr}\nolimits^{1/2}(\bm{B}\bm{C}\bm{B}(\bm{I}\oplus\bm{C}))=\sqrt{12}$.

After calculating the matrix
$$
\bm{B}\bm{C}\bm{B}^{2}
=
\left(
\begin{array}{cccc}
6 & 36 & 12 & 8
\\
2 & 12 & 4 & 1 
\\
6 & 36 & 12 & 3
\\
4 & 24 & 8 & 16/3
\end{array}
\right)
$$
and the trace $\mathop\mathrm{tr}\nolimits^{1/3}(\bm{B}\bm{C}\bm{B}^{2})=\sqrt[3]{12}$, we conclude that
$$
\theta
=
4.
$$

We now form the matrices
\begin{align*}
\theta^{-1}\bm{B}\oplus\bm{C}
&=
\left(
\begin{array}{cccc}
1/4 & 3/4 & 1/2 & 1
\\
1/12 & 1/4 & 1/12 & 1
\\
1/8 & 1 & 1/4 & 1/16
\\
1/16 & 1/2 & 1 & 1/4
\end{array}
\right),
\\
(\theta^{-1}\bm{B}\oplus\bm{C})^{2}
&=
\left(
\begin{array}{cccc}
1/16 & 1/2 & 1 & 3/4
\\
1/16 & 1/2 & 1 & 1/4
\\
1/12 & 1/4 & 1/12 & 1
\\
1/8 & 1 & 1/4 & 1/2
\end{array}
\right),
\\
(\theta^{-1}\bm{B}\oplus\bm{C})^{3}
&=
\left(
\begin{array}{cccc}
1/8 & 1 & 3/4 & 1/2
\\
1/8 & 1 & 1/4 & 1/2
\\
1/16 & 1/2 & 1 & 1/4
\\
1/12 & 1/4 & 1/2 & 1
\end{array}
\right).
\end{align*}

Finally, consider the matrix
$$
(\theta^{-1}\bm{B}\oplus\bm{C})^{\ast}
=
\left(
\begin{array}{cccc}
1 & 1 & 1 & 1
\\
1/8 & 1 & 1 & 1
\\
1/8 & 1 & 1 & 1
\\
1/8 & 1 & 1 & 1
\end{array}
\right).
$$

The last three columns assign the same score equal to one to all alternatives, and therefore, are of no interest. The first column offers a score vector $\bm{x}=(1,1/8,1/8,1/8)^{T}$, which is consistent with both the results of pairwise comparisons, offered by Example~\ref{E-Areciprocal}, and the constraint $x_{2}=x_{3}=x_{4}$.
\end{example}

\bibliographystyle{utphys}

\bibliography{Rating_alternatives_from_pairwise_comparisons_by_solving_tropical_optimization_problems}

\end{document}